\documentclass[11pt,a4paper]{amsart}
\usepackage{amsmath,amssymb,amsthm, amsfonts,enumerate,mathtools,pinlabel,float,stmaryrd}
\usepackage[mathscr]{eucal}
\usepackage{amsbsy}
\usepackage{bm}
\usepackage{tikz}
\usetikzlibrary{arrows}
\usetikzlibrary{decorations.markings}
\usepackage{graphicx}

\newtheorem{theorem*}{Theorem}

\newtheorem{theorem}{Theorem}[section]
\newtheorem{lemma}[theorem]{Lemma}
\newtheorem{proposition}[theorem]{Proposition}
\newtheorem{corollary}[theorem]{Corollary}

\theoremstyle{definition}

\newtheorem{definition}[theorem]{Definition}

\numberwithin{equation}{section}
\theoremstyle{plain}

\newtheorem{question}{Question}

\numberwithin{equation}{section}

\newcommand{\Mod}{\text{Mod}}

\begin{document}

\title[Groups generated by Dehn Twists along fillings of surfaces]{Groups generated by Dehn Twists along fillings of surfaces}
%    Information for author 1
\author{Rakesh Kumar}
\address{School of Mathematics and Computer Science\\
Indian Institute of Technology Goa\\
At Goa College of Engineering Campus \\
Farmagudi, Ponda-403401, Goa\\
India}
\email{rakesh20232101@iitgoa.ac.in}



\maketitle

 \begin{abstract} Let $S_g$ denote a closed oriented surface of genus $g \geq 2$. A set $\Omega = \{ c_1, \dots, c_d\}$ of pairwise non-homotopic simple closed curves on $S_g$ is called a \emph{filling system} or simply a \emph{filling} of $S_g$, if $S_g\setminus \Omega$ is a union of $\ell$ topological discs for some $\ell\geq 1$. For $1\leq i\leq d$, let $T_{c_i}$ denotes the Dehn twist along $c_i$.
In this article, we show that for each $d\geq 2$, there exists a filling $\Omega=\{c_1,c_2,\dots, c_d\}$ of $S_g$ such that the group $\langle T_{c_1}, T_{c_2},\dots,T_{c_d}\rangle$ is isomorphic to the free group of rank $d$. 
\end{abstract}

\section{Introduction}
Let $S_g$ denote a closed oriented surface of genus $g\geq 2$ and $\Omega = \{ c_1, \dots, c_d\}$ is
a nonempty collection of pairwise non-homotopic simple closed curves on $S_g$ such that $i(c_i, c_j )= \lvert c_i \cap c_j \rvert$ for $i \neq j$, i.e., $c_i$ and $c_j$  are in minimal position. Here, $i(c_i, c_j)$ denotes the geometric intersection number of the closed curves $c_i$ and $c_j$. The set $\Omega$ is called a filling of the surface $S_g$ if $S_g \setminus \cup_{i=1}^d c_i$ is a disjoint union of topological discs. The \emph{size} of a filling is defined as the number of its elements. If the number of curves in $\Omega$ is two, then we say that $\Omega$ is a filling pair. For more details on filling of surfaces see   \cite{BS}.\\
Let $\Mod(S_g)$ denote the mapping class group of $S_g$ which is the group of orientation-preserving self-homeomorphisms of $S_g$ up to isotopy. For a simple closed curve $c$ let, $T_c$ denotes the Dehn twist along $c$ on $S_g$. It is well known that $\Mod(S_g)$ is generated by finitely many Dehn twists along non-separating curves \cite{LR,HS}. It is a fact that if $a,b$ are two disjoint essential simple closed curves then $\langle T_a, T_b \rangle$ is isomorphic to $\mathbb{Z} \times \mathbb{Z}$ and if $i(a,b)=1$ then $T_a, T_b$ satisfies the braid relation $T_a T_b T_a =T_b T_a T_b$ \cite{FM}. In this direction when $i(a,b) \geq 2$, Ishida \cite{MR1435728} has proved the following result.
\begin{theorem} \label{T1} 
Let $a$ and $b$ be two isotopy classes of simple closed curves in a surface $S_g$. If $i(a,b) \geq 2$, then the group generated by $T_a$ and $T_b$ is isomorphic to the free group $F_2$ of rank 2.
\end{theorem}
If $\{a,b\}$ is a filling pair on a surface of genus $g\geq 2$ with $\ell$ boundary components, then $i(a,b)=2g-2+\ell \geq 2$ and it follows from  Theorem~\ref{T1} that the group generated by $T_a$ and $T_b$ is isomorphic to the free group of rank 2. One can ask the following question:
\begin{question}
 For $d \geq 3$, does there exist a filling $\Omega = \{ c_1, \dots, c_d\}$ of size $d$ such that the group $\langle T_{c_1}, T_{c_2},\dots,T_{c_d}\rangle$ is isomorphic to the free group of rank $d$.   
\end{question}
We answer this question in the affirmative in the form of the following result:
\begin{theorem} \label{main}
Let $g\geq 2$ and let $S_g$ be a closed oriented surface of genus $g$. For each $d\geq 2$, there exists a filling $\Omega=\{c_1,c_2,\dots, c_d\}$ of $S_g$ such that the group $\langle T_{c_1}, T_{c_2},\dots,T_{c_d}\rangle$ is isomorphic to the free group of rank d.
\end{theorem}
 \section{Preliminaries}\label{prel}
In this section, we define the notion of the Dehn twist and state some basic results that will be instrumental in proving Theorem 1.2.
\begin{definition}
Let $c$  be a simple closed curve on $S_g$. Take a closed annular neighborhood $\mathcal{N}$ of $c$ and identify it with $S^1 \times [0,1]$. Then the Dehn twist along $c$
is isotopy class of the homeomorphism  $T_{c}:S_g \to S_g $ defined by\\
\begin{equation*}
 T_{c}(x) =   \left\{
        \begin{array}{ll}
            x &  if \quad x \notin \mathcal{N} \\
            (e^{2\pi i(\theta+t)}, t) \quad if \quad & \quad x = (e^{2\pi i\theta}, t) \in \mathcal{N} \cong S^1 \times [0,1].
        \end{array}
    \right.
\end{equation*}    
\end{definition}
\begin{proposition} \label{p2} 
Let $a_1,\ldots,a_n$ be a collection of pairwise disjoint isotopy
classes of simple closed curves in a surface $S_g$, and let
$M=\prod_{i=1}^n T^{e_i}_{a_i}$.
Suppose that $e_i>0$ for all $i$ or $e_i<0$ for all $i$. If $b$ and $c$
are arbitrary isotopy classes of simple closed curves in $S_g$, then
$$  |i(M(b),c) - \sum_{i=1}^n |e_i| \cdot |i(a_i,b)| \cdot |i(a_i,c)| \leq i(b,c).$$
\end{proposition}
For a proof see \cite{FM,FLP}. As a special case of Proposition~\ref{p2} $M= T_{a}^{k}$ and $b=c$, we have:
\begin{corollary} \label{c1}

Let a and b be arbitrary isotopy classes of essential simple closed curves in a surface, and let k be an arbitrary integer. We have
$$i(T_{a}^{k}(b),b)=\lvert k \rvert i(a,b)^2$$
\end{corollary}

\textbf{Criterion for filling of a surface } A collection $\Omega=\{c_1, \cdots, c_d \}$ of simple closed curves in a surface $S_g$ is a filling if for every essential simple closed curve $c$  at least one of $i(c,c_i)$ is non-zero.

\textbf{Basic facts about Dehn Twists}
\begin{enumerate}

\item For any $f \in \text{Mod}(S_g)$ and any isotopy class $a$ of simple closed curves in $S$, we have:\\
$$T_{f(a)} = fT_a f^{-1}.$$

\item \textbf{Power of Dehn Twists } For any $f \in \text{Mod}(S_g)$ and any isotopy class $a$ of simple closed curves in $S$, we have:\\
$$fT^{j}_{a} f^{-1} =T^{j}_{f(a)}.$$
\end{enumerate}
The Ping pong argument was first used by Klein \cite{KF}. We will use the following version of ping pong lemma for $n \geq$ 2 players.
\begin{lemma}(\textbf{{Ping pong lemma}}) \label{ppl} 
Let G be a group acting on a set $X$. Let $g_1,... ,g_n$ be elements of $G$. Suppose that there are nonempty, disjoint subsets $X_1,... ,X_n$ of $X$ with the property that, for each $i$ and each $j \neq i$, we have $g^{k}_{i}(X_j ) \subseteq X_i$
for every nonzero integer $k$. Then the group generated by the $\langle g_1,... ,g_n\rangle$ is a free group of rank n.
\end{lemma}
For a proof see \cite[Chapter 3]{FM}.

\section{Main Theorem}
In order to prove Theorem 1.2, we need the following lemma:
\begin{lemma} \label{lemma1}
If $F_2$ is a free group of rank $2$ generated by $f$ and $h$, then $G=\left\langle  h^{-i} f h^{i}; \quad 0 \leq i \leq d-1\right\rangle$ is a free group of rank $d$.
\end{lemma}
\begin{proof}
Let $g_i=h^{-i} f h^{i}$ for $0 \leq i \leq d-1$ be elements of $G$. With the Ping-pong lemma in mind, we define $X_i$ as follows:\\
$X_n$= set of all reduced word starting with $h^{-n}f$ or $h^{-n}f^{-1}$ for $1\leq n \leq d$. Then it is easy to see that $X_n$ are mutually disjoint. Consider the action of $F_2$ on itself by left multiplication. Let $(h^{-j}f^{\pm 1}) \cdot w \in X_j$ . Let 
For $i\neq$j and $k\neq0$, we can see that $(h^{-i} f h^{i})^{k}(h^{-j}f^{\pm 1}) \cdot w=h^{-i} f^{k} h^{i}h^{-j}f^{\pm 1} \cdot w
=h^{-i}f^{k}\cdot w^{'} $ where $w^{'}= h^{i-j}f^{\pm 1} \cdot w$. By definition of $X_i$ we have $h^{-i}f^{k}\cdot w^{'} \in X_i$
which implies $g^{k}_{i}(X_j ) \subseteq X_i $. It follows from Lemma~\ref{ppl} that $G$ is a free group of rank $d$.
\end{proof}
We are now ready to prove Theorem 1.2.
\begin{proof}
Let $S_g$ be a closed oriented surface of genus $g\geq 2$  and $\{a,b \}$ is a filling pair on $S_g$.\\
\textbf{Claim:} $\Omega= \{a, {T^{-1}_{b}(a)}, {T^{-2}_{b}(a)}, \cdots {T^{-(d-1)}_{b}(a)} \}$ is a filling of $S_g$ such that $ G=<T_{a}, T_{T^{-1}_{b}(a)}, T_{T^{-2}_{b}(a)} \cdots T_{T^{-(d-1)}_{b}(a)} > $ is a free group of rank $d$. We will first prove that $\Omega$ is filling of $S_g$. Let $c$ be an arbitrary essential simple closed curve in $S_g$. In order to prove that $\Omega$ is a filling of $S_g$ we need to show that at least one of $i(c, T^{-k}_{b}(a)), 0 \leq k \leq d-1$ is non-zero.
If $i(a,c)\neq 0$ then there is nothing to prove. Now suppose $i(a,c)=0$. In this case we will show that $i(c, T^{-k}_{b}(a))\neq 0 \text{ for } 1 \leq k \leq d-1$. In the setting of Proposition ~\ref{p2} by taking $M= T^{-k}_{b}, b=a, e_{i}=-k $, we get:
 $$ |i(T^{-k}_{b}(a),c) -  |-k| \cdot |i(b,a)| \cdot |i(b,c)| \leq i(a,c)=0$$ which implies  $i(T^{-k}_{b}(a),c)  = ki(b,a) \cdot i(b,c)$. It is easy to see that $i(b,a)\cdot i(b,c)$ is non-zero as $\{a, b\}$ is filling and $i(a,c)=0$.   Hence $i(T^{-k}_{b},c)\neq 0$.\\
Hence $\Omega= \{a, {T^{-1}_{b}(a)}, {T^{-2}_{b}(a)}, \cdots {T^{-(d-1)}_{b}(a)} \}$ is a filling of $S_g$. \\
It remains to prove that group generated by Dehn twists along $\Omega$ is a free group of rank $d$.
As $T_{f^{-1}(a)}= f^{-1}T_af $, then 
 $G=\left\langle  T_{b}^{-i} T_{a} T_{b}^{i}; \quad 0 \leq i \leq d-1\right\rangle$.
Suppose the filling pair $\{a,b\}$ has $\ell$ boundary components. The Euler's characteristic equation implies $i(a,b)=2g-2+\ell \geq 2$. It follows Theorem~\ref{T1} that the group generated by $T_a$ and $T_b$ is isomorphic to the free group of rank 2. Then it follows from lemma~\ref{lemma1} that $G=\left\langle  T_{b}^{-i} T_{a} T_{b}^{i}; \quad 0 \leq i \leq d-1\right\rangle$ is isomorphic to the free group of rank $d$.

\end{proof}
Finally, we will count the number of boundary components of filling constructed in the proof of Theorem~\ref{main}.

\begin{proposition}
Let $S_g$ be a closed oriented surface of genus $g\geq 2$  and $\{a,b \}$ is a filling pair on $S_g$. The number of boundary components of the surface $S_g$ for the filling $\Omega=  \{a, {T^{-1}_{b}(a)}, {T^{-2}_{b}(a)}, \cdots {T^{-(d-1)}_{b}(a)} \}$ is $i(a,b)^{2}  \frac{d(d-1)(d+1)}{6}-2g+2$.
\end{proposition}
\begin{proof}
Let $a=c_1, {T^{-1}_{(b)}(a)}=c_2, \cdots ,{T^{-(d-1)}_{(b)}(a)}=c_d$, so $\Omega=\{ c_1, c_2, \cdots, c_d\}$. Euler's characteristic equation gives $\sum_{i\neq j}i(c_i, c_j)=2g-2+\ell$, here $\ell$ denotes the number of boundary components.\\
Then it follows from corollary~\ref{c1} $i(T_{a}^{k}(b),b)=\lvert k \rvert i(a,b)^2 $,
$$i(a,T_{b}^{-1}(a))=i(a,b)^2,$$
$$i(a,T_{b}^{-2}(a))=\lvert -2 \rvert i(a,b)^2 \implies i(a,T_{b}^{-2}(a))=2i(a,b)^2 $$ 
\hspace{2.5cm}$\vdots $ \hspace{2.5cm}$\vdots $\\
$$i(a,T_{b}^{-(d-1)}(a))=\lvert -(d-1) \rvert i(a,b)^2 \implies i(a,T_{b}^{-(d-1)}(a))=(d-1)i(a,b)^2. $$
 $$ \text{Similarly} \quad i(T_{b}^{-1}(a),T_{b}^{-2}(a))=i(a,T_{b}^{-1}(a)) =i(a,b)^2,$$
$$i(T_{b}^{-1}(a),T_{b}^{-3}(a))=i(a,T_{b}^{-2}(a)) =2i(a,b)^2,$$
\hspace{3.4cm}$\vdots $ \hspace{3.4cm}$\vdots $\\
$$i(T_{b}^{-1}(a),T_{b}^{-(d-1)}(a))=i(a,T_{b}^{-(d-2)}(a)) =(d-2)i(a,b)^2.$$
$$ \text{Similarly} \quad i(T_{b}^{-2}(a),T_{b}^{-3}(a))=i(a,b)^2,$$
$$i(T_{b}^{-2}(a),T_{b}^{-4}(a))=2i(a,b)^2,$$
\hspace{4.4cm}$\vdots $ \hspace{3.8cm}$\vdots $\\
$$i(T_{b}^{-2}(a),T_{b}^{-(d-1)}(a))=(d-3)i(a,b)^2.$$
\hspace{4.4cm}$\vdots $ \hspace{3.8cm}$\vdots $

$$i(T_{b}^{-(d-2)}(a),T_{b}^{-(d-1)}(a))=i(a,T_{b}^{-1}(a))= i(a,b)^2.$$
Then $\sum_{i\neq j}i(c_i, c_j)=(i(a,b)^2(1+2+ \cdots +(d-1)+i(a,b)^2(1+2+ \cdots +\cdots (d-2))+i(a,b)^2(1+2\cdots +(d-3))+ \cdots \cdots + i(a,b)^2(1)).$\\ $\sum_{i\neq j}i(c_i, c_j)=(i(a,b)^2)(1+(1+2)+(1+2+3)+ \cdots + (1+2+ \cdots + (d-1))$,
$\sum_{i\neq j}i(c_i, c_j)=(i(a,b)^2)\cdot \frac{d(d-1)(d+1)}{6}$, then it follows from Euler's characteristic equation $2g-2+\ell=i(a,b)^2 \frac{d(d-1)(d+1)}{6} \quad \text{that} \quad \ell=i(a,b)^{2} \frac{d(d-1)(d+1)}{6}-2g+2.$

\end{proof}

\section*{Acknowledgements}
The author would like to thank Council of Scientific and Industrial Research (CSIR)( 09/1290(0002)2020-EMR-I ) for providing financial support.

\bibliographystyle{plain}      
\bibliography{Reference.bib}  

\begin{thebibliography}{1}

\bibitem{FM}
Benson Farb and Dan Margalit.
\newblock {\em A primer on mapping class groups}, volume~49 of {\em Princeton
  Mathematical Series}.
\newblock Princeton University Press, Princeton, NJ, 2012.

\bibitem{FLP}
Albert Fathi, Fran\c{c}ois Laudenbach, and Valentin Po\'{e}naru.
\newblock {\em Thurston's work on surfaces}, volume~48 of {\em Mathematical
  Notes}.
\newblock Princeton University Press, Princeton, NJ, 2012.
\newblock Translated from the 1979 French original by Djun M. Kim and Dan
  Margalit.

\bibitem{HS}
Stephen~P. Humphries.
\newblock Generators for the mapping class group.
\newblock In {\em Topology of low-dimensional manifolds ({P}roc. {S}econd
  {S}ussex {C}onf., {C}helwood {G}ate, 1977)}, volume 722 of {\em Lecture Notes
  in Math.}, pages 44--47. Springer, Berlin, 1979.

\bibitem{MR1435728}
Atsushi Ishida.
\newblock The structure of subgroup of mapping class groups generated by two
  {D}ehn twists.
\newblock {\em Proc. Japan Acad. Ser. A Math. Sci.}, 72(10):240--241, 1996.

\bibitem{KF}
Felix Klein.
\newblock Neue {B}eitr\"{a}ge zur {R}iemann'schen {F}unctionentheorie.
\newblock {\em Math. Ann.}, 21(2):141--218, 1883.

\bibitem{LR}
W.~B.~R. Lickorish.
\newblock A finite set of generators for the homeotopy group of a
  {$2$}-manifold.
\newblock {\em Proc. Cambridge Philos. Soc.}, 60:769--778, 1964.

\bibitem{BS}
Bidyut Sanki.
\newblock Filling of closed surfaces.
\newblock {\em J. Topol. Anal.}, 10(4):897--913, 2018.

\end{thebibliography}
\end{document}